\documentclass{amsart}
\usepackage{amssymb}
\usepackage{amscd}
\usepackage{amsmath}
\usepackage{epsfig}
\usepackage{graphicx}

\usepackage[all]{xy}

\usepackage{pgf,tikz}
\usepackage{caption}
\usepackage{shuffle}

\usepackage{time}
 
\usetikzlibrary{arrows}
\usetikzlibrary{decorations.pathreplacing}

\title[Pentagon and Confluence]
{The pentagon equation and\\ the confluence relations}
\author{Hidekazu Furusho}
\address{Graduate School of Mathematics, Nagoya University, Chikusa-ku, Furo-cho, Nagoya, 464-8602,  Japan}
\email{furusho@math.nagoya-u.ac.jp}



\date{February 13, 2022}

\newtheorem{thm}{Theorem}
\newtheorem{lem}[thm]{Lemma}
\newtheorem{cor}[thm]{Corollary}
\newtheorem{prop}[thm]{Proposition}  

\theoremstyle{remark}

\theoremstyle{definition}
\newtheorem{defn}[thm]{Definition}
\newtheorem{rem}[thm]{Remark}

\numberwithin{equation}{section}
\numberwithin{figure}{section}

\newcommand{\Q}{\mathbb{Q}}
\newcommand{\C}{\mathbb{C}}
\newcommand{\R}{\mathbb{R}}
\newcommand{\Z}{\mathbb{Z}}
\newcommand{\N}{\mathbb{N}}

\newcommand{\A}{\mathcal{A}}
\newcommand{\B}{\mathcal{B}}

\newcommand{\M}{\mathcal{M}}
\newcommand{\X}{\mathcal{X}}
\newcommand{\Y}{\mathcal{Y}}
\newcommand{\I}{\mathcal{I}}

\newcommand{\Const}{\mathrm{Const}}
\newcommand{\pr}{\mathrm{pr}}
\newcommand{\id}{\mathrm{id}}
\newcommand{\reg}{\mathrm{reg}}

\newcommand{\dec}{\mathrm{dec}}
\newcommand{\ST}{\mathrm{ST}}
\newcommand{\CF}{\mathrm{CF}}

\newcommand{\Li}{\mathrm{Li}}

\newcommand{\sha}{\shuffle}
            
\begin{document}
\bibliographystyle{amsalpha+}
\maketitle

\begin{abstract}
We show an equivalence of 
Drinfeld's pentagon equation and Hirose-Sato's confluence relations.
As a corollary, we obtain a \lq pentagon-free' presentation of 
the Grothendieck-Teichm\"{u}ller group $GRT_1$
and associators.
\end{abstract}

\tableofcontents
\section{Introduction}\label{introduction}
This paper discusses an equivalence of two types of relations (Theorem \ref{thm:main}):
\begin{itemize}
\item
Drinfeld's pentagon equation (\cite{Dr}) which is the main defining equation of associators and the  Grothendieck-Teichm\"{u}ller group.
\item
Hirose-Sato's confluence relations  (\cite{HS}) which are
(conjectured to exhaust all the)
relations among
multiple zeta values (cf. \S\ref{sec:confluence relation}).
\end{itemize}
As a corollary, a new formulation of associators and the Grothendieck-Teichm\"{u}ller group is obtained (Corollary \ref{cor:main}).

Let $\frak f_2$ be the free Lie algebra over $\Q$ with two variables $f_0$ and $f_1$
and $U\frak f_2:=\Q\langle f_0,f_1\rangle$ be its universal enveloping algebra.
We denote $\widehat{\frak f_2}$ and $\widehat{U\frak f_2}:=\Q\langle\langle f_0,f_1\rangle\rangle$ to be their completions by
degrees. 
An {\it associator} (\cite{Dr, F10}) is
a series $\varphi=\varphi(f_0,f_1)$ in
$\widehat{U\mathfrak f_2}$
with non-zero quadratic terms
which satisfies the following:
\begin{itemize}
\item the commutator group-like condition: $\varphi\in \exp [\hat{\mathfrak f}_2,\hat{\mathfrak f}_2].$
\item the pentagon equation:
$\varphi_{345}\varphi_{512}\varphi_{234}\varphi_{451}\varphi_{123}=1$
in  $\widehat{U\mathfrak P_5}$.
\end{itemize}
Here $\exp [\hat{\mathfrak f}_2,\hat{\mathfrak f}_2]$ is the image of
the topological commutator $[\hat{\mathfrak f}_2,\hat{\mathfrak f}_2]$
of $\hat{\mathfrak f}_2$
under the exponential map and
$\frak P_5$ is the Lie algebra generated by $t_{ij}$ ($i,j\in\Z/5$)
with the relations
\begin{itemize}
\item $t_{ij}=t_{ji}$, \quad $t_{ii}=0$,
\item $\sum\nolimits_{j\in\Z/5}t_{ij}=0\quad  (\forall i\in \Z/5)$,
\item $[t_{ij},t_{kl}]=0 \quad \text{for}\quad \{i,j\}\cap\{k,l\}=\emptyset$.
\end{itemize}
For  $i,j,k\in\Z/5$,  $\varphi_{ijk}$ means the image of $\varphi$ under the 
embedding $\widehat{U\mathfrak f_2}\to\widehat{U\mathfrak P_5}$
sending $f_0\mapsto t_{ij}$ and $f_1\mapsto t_{jk}$.

\begin{thm}\label{thm:main}
Let $\varphi$ be a commutator group-like series in $\widehat{U\frak f_2}$,
i.e. $\varphi\in \exp [\hat{\mathfrak f}_2,\hat{\mathfrak f}_2].$
Then it satisfies the pentagon equation if and only if it satisfies the confluence relations
(cf. \S\ref{sec:confluence relation}). 
\end{thm}

As a corollary, we obtain
a new formulation of the set $M(\Q)$ of associators
and also that of the graded Grothendieck-Teichm\"{u}ller  group $GRT_1(\Q)$,
the set  of \lq associators without quadratic terms' (\cite{Dr}):
\begin{cor}\label{cor:main}
There are equalities:
\begin{align*}
M(\Q) &=\{\varphi\in \exp [\hat{\mathfrak f}_2,\hat{\mathfrak f}_2] \mid \
<e_0e_1|\varphi>\neq 0
\text{ and }<l|\varphi>=0 \text{ for }  l\in\I_\CF\}, \\
GRT_1(\Q) &=\{\varphi\in \exp [\hat{\mathfrak f}_2,\hat{\mathfrak f}_2] \mid \
<e_0e_1|\varphi>= 0
\text{ and }<l|\varphi>=0 \text{ for }  l\in\I_\CF\}. 
\end{align*}
\end{cor}
For the set $\I_\CF$ of confluence relations, consult \S\ref{sec:confluence relation}
and for the pairing $<\cdot | \cdot>$, see \eqref{eqn:pairing}.
It is remarkable that
the right hand side of the second equation turns to be a group under the operation
$$\varphi_1\circ\varphi_2=\varphi_1(\varphi_2(f_0,f_1) f_0\varphi_2(f_0,f_1)^{-1},f_1)\varphi_2(f_0,f_1)$$
because $GRT_1(\Q)$ forms a group (\cite{Dr}).
Its associated Lie algebra $\frak{grt}_1(\Q)$ is described as
$$
\frak{grt}_1(\Q)=\{\psi\in [\hat{\mathfrak f}_2,\hat{\mathfrak f}_2] \mid \
<e_0e_1|\psi>= 0
\text{ and }<l|\psi>=0 \text{ for }  l\in\I_\CF\}.
$$

It might be worthy to pursue a pentagon-free presentation of
the filtered version $GT_1(\Q)$ and the profinite version $\widehat{GT}_1$
of the Grothendieck-Teichm\"{u}ller group which were both introduced in \cite{Dr}.

In \cite{HS} they showed that the confluence relations imply the regularized (generalized) double shuffle relations (cf. Theorem \ref{thm:HS confluence-->double shuffle}).
By combining it with Theorem \ref{thm:main},
we recover the result in \cite{F11} that the associator relation implies 
the generalized double shuffle relations.
The author expects that our new formulation would give us a further understanding on the implication. 
 
Here is the plan of the proof of Theorem \ref{thm:main}:
Under the commutator group-like condition, we  show that the pentagon equation
implies the confluence relations (Theorem \ref{thm:main 1})
and vice versa (Theorem \ref{thm:main 2}).

\section{Confluence relations}\label{sec:confluence relation}
We recall the definition of the confluence relations 
by employing their original symbols in \cite{HS}.

Let $z\neq 0,1 \in \C$. 
Put $\A_z=\Q\langle e_0,e_z,e_1\rangle$.
We consider the following sequence of linear subspaces
$\A_z\supset \A_z^1\supset \A_z^0\supset\A_z^{-1} \supset \A_z^{-2}$,
where
\begin{align*}
&\A_z^1 =\Q\oplus\A_z e_1\oplus \A_z e_z, \\
&\A_z^0 =\Q\oplus \Q e_z \oplus_{a\in\{0,z\}, b\in\{1,z\}} e_a\A_z e_b, \\
&\A_z^{-1} =\Q\langle e_z\rangle\cdot \A_z^{-2},\\
&\A_z^{-2} =\Q\oplus e_0\A_ze_1 \oplus e_0\A_ze_z.
\end{align*}
(For our purpose, we reverse the orders in their definitions of the subspaces in \cite{HS},
that is, we read them backwards.)
For $\A=\Q\langle e_0,e_1\rangle$,
set $\A^i:=\A\cap \A^i_z$. We have  $\A^{-2}=\A^{-1}=\A^0$.
All the spaces form algebras under the shuffle product $\sha:\A_z\times \A_z \to \A_z$ such that
$e_au\sha e_bv=e_a(u\sha e_bv)+e_b(e_au\sha v)$ ($u,v\in\A_z$).
By abuse of notation, we also consider their copies
$\A_w=\Q\langle e_0,e_w,e_1\rangle
\supset \A_w^1\supset \A_w^0\supset\A_w^{-1} \supset \A_w^{-2}$
for $w\neq 0,1 \in \C$.

We consider the iterated integral given by
\begin{equation}\label{eqn:L}
L(e_{a_n}\cdots e_{a_1}):=\int_{0<t_1<\cdots<t_n<1}
\frac{dt_n}{t_n-a_n}\wedge\cdots\wedge\frac{dt_1}{t_1-a_1}
\end{equation}
and $L(1)=1$.
It yields a linear map 
$L:\A^0_z\to \mathrm{Hol}_z(\C\setminus [0,1])$
to the space of holomorphic functions on
$\C\setminus [0,1]$.

The image is a holomorphic function on $z$
(Its relation with multiple polylogarithm will be discussed in Remark \ref{rem:rho and MPL}). 
Particularly 
the image of $\A^0$ is given by a multiple zeta value (MZV in short):
$$L(e_0^{k_1-1}e_1\cdots e_0^{k_m-1}e_1)=(-1)^m\zeta(k_1,\dots,k_m)$$
with
$\zeta(k_1,\dots,k_m):=\sum_{n_1>\cdots>n_m>0}
\frac{1}{n_1^{k_1}\cdots n_m^{k_m}}$ ($k_1>1$).

The shuffle algebra homomorphism 
$\Const: \A_z\to \A$ is defined by 
$\Const(e_{a_1}\cdots e_{a_n})=e_{a_1}\cdots e_{a_n}$ if $a_i\neq z$ for all $i$,
and 0 otherwise.
We have $\lim_{z\to \infty}L(l)=L(\Const(l))$ for $l\in\A^0_z$.

The linear operator $\partial_{z,\alpha}: \A_z\to \A_z$  ($\alpha=0,1$)
is defined by 
\begin{equation}\label{eqn:partial z alpha}
\partial_{z,\alpha}(e_{a_n}\cdots e_{a_1})
=\sum\nolimits_{i=1}^n\left(\delta_{\{a_i,a_{i+1}\},\{z,\alpha\}}-\delta_{\{a_{i-1},a_{i}\},\{z,\alpha\}}\right)
e_{a_n}\cdots \Check{e}_{a_i}\cdots e_{a_1}
\end{equation}
and $\partial_{z,\alpha}(1)=0$
with $a_0=0$ and $a_{n+1}=1$.
Here $\delta_{\{\alpha_1,\alpha_2\},\{\beta_1,\beta_2\}}$ is the Kronecker delta function, i.e. 1 if $\{\alpha_1,\alpha_2\}=\{\beta_1,\beta_2\}$ as sets and $0$ otherwise.
It is shown in \cite{HIST} that
$$\frac{d}{dz}L(l)=\frac{1}{z}L(\partial_{z,0}(l))+\frac{1}{z-1}L(\partial_{z,1}(l)).$$

The set of  standard relations  is defined to be the subspace of $\A^0_z$,
$$
{\I}_{\ST}:=\{l\in\A^0_z \mid 
\Const(\partial_{z,\alpha_1}\cdots\partial_{z,\alpha_r}l)=0
\text{ for } 
r\geqslant 0, \alpha_1,\dots,\alpha_r\in\{0,1\}
\},
$$
which actually forms an ideal of $(\A^0_z,\sha)$.
An element of ${\I}_{\mathrm{ST}}$ is called {\it a standard relation}.
They  showed that $L(\I_\ST)=0$.

They considered 
the algebra homomorphism
$N:(\A^0_z,\sha)\to (\A^{-1}_z,\sha)$
which is defined by the composition of the following algebra homomorphisms
with respect to $\sha$
$$
N:
\A^0_z
\overset{\reg_{z,1}}{\underset{\simeq}{\longrightarrow}}
\A_z^{-2}\otimes (\A^0_z\cap \Q\langle e_1,e_z\rangle)
\overset{\id\otimes{\tau_z}}{\longrightarrow}
\A_z^{-2}\otimes (\A^0_z\cap \Q\langle e_0,e_z\rangle)
{\twoheadrightarrow}
\A_z^{-1}.
$$
Here $\reg_{z,1}$  in the first map is caused by the isomorphism
$\A_z^{-2}\otimes (\A^0_z\cap \Q\langle e_1,e_z\rangle) \simeq \A_z^0;
u\otimes v\mapsto u\sha v$,
$\tau_z$ appearing in the second map is the anti-automorphism
${\tau_z}: (\A_z,\cdot)\to (\A_z,\cdot)$ such that 
$\begin{cases}
e_0 \mapsto & e_z-e_1,\\
e_1 \mapsto &e_z-e_0,\\
e_z \mapsto &e_z,
\end{cases}$
and the third map is the surjection simply induced by 
the shuffle algebra homomorphism $\A_z^{\otimes 2}\to \A_z;$
$u\otimes v\mapsto u\sha v$.
Then they introduced the algebra homomorphism
$\lambda:(\A^0_z,\sha)\to (\A^0,\sha)$
which is defined by the composition of the following algebra homomorphisms
$$
\lambda:
\A^0_z\overset{N}{\to}\A^{-1}_z
\overset{\reg_z}{\underset{\simeq}{\longrightarrow}}
\Q\langle e_z\rangle\otimes \A_z^{-2}
\overset{\mathrm{const}\otimes \id}{\longrightarrow}
\A_z^{-2}
\overset{|_{z\to 1}}{\longrightarrow}
\A^0.
$$
Here $\reg_z$ is caused by the isomorphism
$\Q\langle e_z \rangle\otimes \A_z^{-2} \simeq \A_z^{-1};
u\otimes v\mapsto u\sha v$,
$\mathrm{const}=\Const|_{\Q\langle e_z\rangle}$ and
$|_{z\to 1}:\A_z^{-2}\to \A^0; l\mapsto l_{z\to 1}$
is the algebra homomorphism
sending 
$e_z\mapsto e_1$ and $e_i\mapsto e_i$ for $i=0,1$.

The set of confluence relations is defined to be the image of $\I_\ST$
in $\A^0$
under the map $\lambda$,
$$\I_\CF:=\lambda(\I_\ST).$$
An element of $\I_\CF$ is called a {\it confluence relation}.

\begin{thm}[\cite{HS}]\label{thm:HS confluence-->double shuffle}
(i). We have $L(\I_\CF)=0$, that is, the confluence relations give linear relations among MZV's. \\
(ii). The confluence relations imply the regularized (generalized) double shuffle relations 
and also the duality relation, namely we have
$\I_\CF\supset \I_{\mathrm{RDS}}$ and $\I_\CF\supset \I_{\Delta}$.
\end{thm}

Here $\I_{\mathrm{RDS}}$  (resp. $\I_\Delta$) is the ideal of $(\A^0,\sha)$ generated by
$\{\reg_\sha(u\sha v-u \ast v)\mid u\in\A^1, v\in\A^0\}$
(resp. $\{w-\tau_\infty(w)\mid w\in\A^0\}$
where $\tau_\infty$ is the anti-automorphism of $\A$ sending $\tau_\infty(e_0)=-e_1$, $\tau_\infty(e_1)=-e_0$)
(for definitions of $\reg_\sha$ and $\ast$, consult \cite{HS}).

In \cite{HS} Conjecture 24, it is conjectured that  the confluence relations
exhaust all the relations among MZV's, i.e. $\I_\CF=\ker\{L:\A^0\to\R\}$.

\section{The pentagon equation implies the confluence relations}
We prove that the pentagon equation implies the confluence relation under the commutator group-like condition
(Theorem \ref{thm:main 1}).

We regard $U\frak f_2=\Q\langle f_0,f_1\rangle$ as the dual of 
$\A=\Q\langle e_0,e_1\rangle$ 
by the pairing
\begin{equation}\label{eqn:pairing}
<\cdot|\cdot>:\A \otimes U\frak f_2\to \Q
\end{equation}
such that
$<W_1(e_0,e_1) \mid  W_2(f_0,f_1)>=\delta_{W_1,W_2}$ for any words $W_1,W_2$.
Here a word means a monic monomial element of the free monoid  generated by two elements, say $A$ and $B$.
For a word $W$, we denote $W(e_0,e_1)$ (resp. $W(f_0,f_1)$) to be the corresponding element in $\A$ (resp. $U\frak f_2$) obtained by the substitution of
$e_0,e_1$ (resp. $f_0,f_1$) to $A, B$.

\begin{thm}\label{thm:main 1}
Let $\varphi\in \exp [\hat{\mathfrak f}_2,\hat{\mathfrak f}_2].$
Assume that $\varphi$ satisfies the pentagon equation.
Then for any $l\in{\mathcal I}_{\CF}$ we have $<l \mid\varphi>=0.$
In other word,
for any $l\in{\mathcal I}_{\ST}$ we have
$<\lambda(l) \mid \varphi>=0.$
\end{thm}

\subsection{Bar construction and Oi-Ueno decomposition}
We prepare basic tools of bar-constructions of varieties,
$\mathcal{M}_{0,5}$, $\mathcal{M}_{0,4}$, $\X(z)$, $\Y(w)$
and recall two decompositions of the reduced bar algebra $\B$ of
$\mathcal{M}_{0,5}$ 
by Oi-Ueno \cite{OU}.

We consider three varieties.
Let $\mathcal{M}_{0,4}:
=\{(p_1,\dots,p_4)\in \mathbb P^1(\Q)^4 \mid p_i\neq p_j\}/ \mathrm{PGL}(2)$
and $\mathcal{M}_{0,5}:=\{(p_1,\dots,p_5)\in \mathbb P^1(\Q)^5 \mid p_i\neq p_j\}/ \mathrm{PGL}(2)$
be the moduli spaces of the projective line $\mathbb P^1(\Q)$ with  $4$ and $5$ marked points respectively.
For $i=1,\dots,5$,
we have  the projections $\pr_i:\M_{0,5}\to \M_{0,4}$ induced by the omission of the $i$-th parametrized point respectively,
which are encoded with the action of the symmetric groups
${\mathfrak S}_4$  and ${\mathfrak S}_5$.
By taking the normalized coordinate $(0,x,1,\infty)$, we identify $\M_{0,4}$ with
the space
$\{x\in \Q \mid x\neq 0,1\}$ and
by taking the normalized coordinate $(0,w,z,1,\infty)$, we identify $\M_{0,5}$ with
$\{(w,z)\in \Q^2 \mid  z,w\neq 0,1, \ z\neq w \}.$
Put $\X(z):=\{w\in\Q \mid w\neq 0,1,z\}$ for a fixed $z$.
We often regard $\X(z)$ be a subspace of  $\M_{0,5}$
under $j_\X:w\mapsto (w,z)$.
Similarly we consider
$\Y(w):=\{z\in\Q \mid z\neq 0,1,w\}$ for a fixed $w$
and regard it as a subspace of $\M_{0,5}$
under $j_\Y:z\mapsto (\frac{1}{w},\frac{1}{z})$.

\begin{rem}
There is a geometric picture which might help intuitively our later arguments:
Our $\M_{0,5}$  is regarded to be $\overline{\M_{0,5}}\setminus \cup_{i\neq j} L_{ij}$.
Here $\overline{\M_{0,5}}$ is the stable compactification of
$\M_{0,5}$,
which is a blowing-up of $\mathbb P^1\times\mathbb P^1$ at
$(z,w)=(0,0)$, $(1,1)$, $(\infty,\infty)$.
Its real structure is depicted in Figure \ref{fig:M05}.
Here $L_{ij}$ means the boundary component given by $p_i=p_j$,
which is isomorphic to $\mathbb P^1(\Q)$.
Particularly $L_{45}$, $L_{15}$, $L_{14}$ are the exceptional divisors.
They are depicted as if they were hexagons but they represent $\mathbb P^1(\Q)$
which identify  interfacing vertices and edges.
We note that our parameter $t_{ij}\in \widehat{U\mathfrak P_5}$ corresponds to the local monodoromy around
$L_{ij}$.
For some 5-tuples $\{i,j,k,l,m\}=\{1,2,3,4,5\}$,
the symbol $\varphi_{klm}$ is depicted aside the oriented edge on $L_{ij}$
connecting the two vertices $v_{ijkl}=L_{ij}\cap L_{kl}$ and $v_{ijlm}=L_{ij}\cap L_{lm}$.
It represents  the image of $\varphi\in \widehat{U\frak f_2}$
under the embedding $\widehat{U\mathfrak f_2}\to\widehat{U\mathfrak P_5}$;
$f_0\mapsto t_{kl}$ and $f_1\mapsto t_{lm}$,
which is caused by the identification $\mathbb P^1(\Q)$ with $L_{ij}$
such that $0\mapsto v_{ijkl}$ and $1\mapsto v_{ijlm}$
and the inclusion $L_{ij}\hookrightarrow \overline{\M_{0,5}}$.
\end{rem}

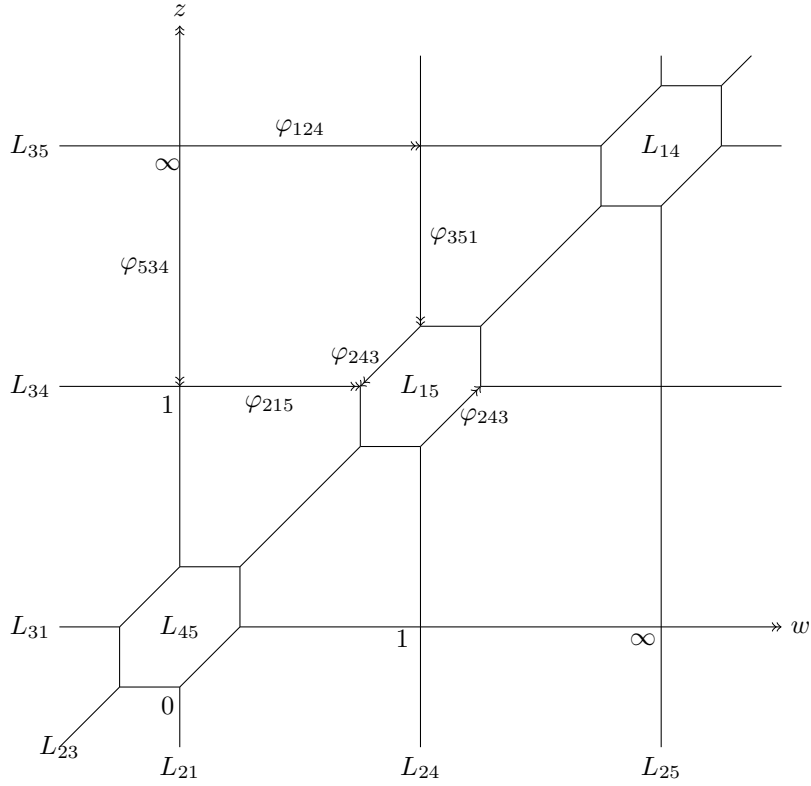
\begin{figure}[h]
\begin{center}
\begin{tikzpicture}[scale=0.8]
  \draw[-] (-1,0)--(-2,0) node[left]{$L_{31}$};

  \draw[-] (0,1)--(1,1); 
  \draw[-] (1,0)--(1,1);
  \draw[-] (1,0)--(0,-1); 
  \draw[-] (0,-1)--(-1,-1); 
  \draw[-] (-1,-1)--(-1,0);
  \draw[-] (-1,-1)--(-2,-2) node{$L_{23}$};
  \draw[-] (-1,0)--(0,1); 
  \draw[-] (0,0) node{$L_{45}$};

  \draw (-0.2,-1.3) node{$0$};
  \draw[-] (0,-1)--(0,-2) node[below]{$L_{21}$}; 
  \draw[-] (0,1)--(0,4); 
  \draw[-] (0,4)--(-2,4) node[left]{$L_{34}$};
  \draw (-0.2,3.7) node{$1$};
  \draw[->>] (0,4)--(3,4) node[midway,below]{$\varphi_{215}$};

  \draw[<<-] (0,4)--(0,8) node[midway,left]{$\varphi_{534}$};
  \draw[-] (0,8)--(-2,8) node[left]{$L_{35}$};
  \draw (-0.2,7.7) node{$\infty$};
  \draw[->>] (0,8)--(0,10) node[above]{$z$};

  \draw[-] (3,3)--(1,1); 
  
  \draw[-] (3,3)--(3,4) ; 
  \draw[<<-] (3,4)--(4,5) node[midway,left]{$\varphi_{243}$};
    \draw  (4,4) node{$L_{15}$} ;

  \draw[-] (4,5)--(5,5) ; 
 \draw[-] (5,5)--(5,4) ; 
 \draw[->>] (4,3)--(5,4) node[midway,right]{$\varphi_{243}$};
 \draw[-] (4,3)--(3,3) ;
 \draw[-] (5,4)--(10,4) ; 
  \draw[-] (5,5)--(7,7) ; 

  \draw[-] (4,0)--(4,-2) node[below]{$L_{24}$}; 
  \draw (3.7,-0.2) node{$1$};
  \draw[-] (4,0)--(4,3) ; 

  \draw[->>] (4,8)--(4,5) node[midway,right]{$\varphi_{351}$}; 
  \draw[-] (4,8)--(4,9.5) ; 
  \draw[->>] (0,8)--(4,8) node[midway,above]{$\varphi_{124}$};
 \draw[-] (4,8)--(7,8) ; 

  \draw[-] (7,7)--(7,8) ;
  \draw[-] (7,7)--(8,7) ;
  \draw[-] (8,7)--(9,8) ;
  \draw[-] (7,8)--(8,9) ; 
  \draw  (8,8) node{$L_{14}$} ;
  \draw[-] (8,9)--(9,9) ;
  \draw[-] (9,8)--(9,9) ;

  \draw[-] (9,8)--(10,8) ;
  \draw[-] (9,9)--(9.5,9.5) ;

  \draw[-] (8,7)--(8,-2) node[below]{$L_{25}$} (8,9)--(8,9.5);
  \draw (7.7,-0.2) node{$\infty$};

  \draw[->>] (1,0)--(10,0) node[right]{$w$} ;

\end{tikzpicture}
\caption{$\mathcal{M}_{0,5}(\mathbb{R})$}
\label{fig:M05}
\end{center}
\end{figure}

For $\M=\mathcal{M}_{0,5}$, $\mathcal{M}_{0,4}$, $\X(z)$ and $\Y(w)$
we consider its
Chen's \cite{C} reduced bar algebra  $H^0\bar B(\Omega^\ast_\mathrm{DR}(\M))$.
It is calculated to be the graded Hopf algebra $\B(\M)=\oplus_{m=0}^\infty \B(\M)_m$
($\subset T\B(\M)_1=\oplus_{m=0}^\infty \B(\M)_1^{\otimes m}$)
over $\Q$,
where
$\B(\M)_0=\Q$, $\B(\M)_1=H^1_{\mathrm{DR}}(\M)$ and
$\B(\M)_m$ is the totality of linear combinations (finite sums) \\
$$\sum_{I=(i_m,\cdots,i_1)}c_I[\omega_{i_m}|\cdots|\omega_{i_1}]
\in \B(\M)_1^{\otimes m}$$
($c_I\in \Q$, $\omega_{i_j}\in \B(\M)_1$,
$[\omega_{i_m}|\cdots|\omega_{i_1}]:=\omega_{i_m}\otimes\cdots\otimes\omega_{i_1}$)
satisfying the integrability condition
$$
\sum_{I=(i_m,\cdots,i_1)}c_I[\omega_{i_m}|\cdots
|\omega_{i_{j+1}}\wedge\omega_{i_{j}}|\cdots|\omega_{i_1}]=0
$$
in $\B(\M)_1^{\otimes m-j-1}\otimes H^2_\mathrm{DR}(\mathcal M)\otimes \B(\M)_1^{\otimes j-1}$
for all $j$ ($1\leqslant j<m$). 
Its product is given by the shuffle product $\sha$ and its coproduct $\delta$
is given by deconcatenation. 

\begin{lem}\label{lem:bar identifications}
We have the following isomorphisms of Hopf algebras:
\begin{align*}
& H^0\bar B(\Omega^\ast_\mathrm{DR}(\M_{0,4}))\simeq \A, \quad
H^0\bar B(\Omega^\ast_\mathrm{DR}(\X(z))\simeq \A_z, \quad 
H^0\bar B(\Omega^\ast_\mathrm{DR}(\Y(w))\simeq \A_w, \quad \\
& H^0\bar B(\Omega^\ast_\mathrm{DR}(\M_{0,5}))\simeq U\frak P_5^*.
\end{align*}
Here $U\frak P_5^*$ is the graded linear dual of $U\frak P_5$.
\end{lem}
The first isomorphism  is given by the correspondence
$d\log x\mapsto e_0$, $d\log(x-1)\mapsto e_1$.
The second one is given by the correspondence
$d\log w\mapsto e_0$, $d\log(w-z)\mapsto e_z$, $d\log(w-1)\mapsto e_1$.
The third one is given by 
$d\log z\mapsto e_0$, $d\log(z-w)\mapsto e_w$, $d\log(z-1)\mapsto e_1$.
Though the last one is explained in \cite{F11},  we further investigate it below.

Put $\B:=H^0\bar B(\Omega^\ast_\mathrm{DR}(\M_{0,5}))$ and
fix the $\Q$-linear basis of $\B(\M_{0,5})_1$ as
$$e_{21}=d\log w, \
e_{23}=d\log (w-z), \
e_{24}=d\log (w-1), \
e_{31}=d\log z, \ 
e_{34}=d\log (z-1).$$
We note that $t_{21}, t_{23}, t_{24}, t_{31}, t_{34}$ forms a basis of
the degree $1$  part of $\widehat{U\frak P_5}$ and
$e_{21}, e_{23}, e_{24}, e_{31}, e_{34}$ gives its dual basis. 

Following the setting in \cite{OU},
we employ the symbols below:
\begin{align*}
&\eta_2   =d\log w= e_{21},   \quad
&\eta_{22} =d\log (1-w) =e_{24},  \phantom{xxxxxxxxxxx}
\\
&\eta_3    =d\log \frac{1}{z}=-e_{31}, 
& \eta_{33} =d\log(1-\frac{1}{z})=-e_{31}+e_{34}, 
\phantom{xxxxx}
\\
&\eta_{23} =d\log(1-\frac{w}{z})=-e_{31}+e_{23},
&\text{and }
\eta_{23}^{(2)}=\frac{-dw}{z-w}, \quad
\eta_{23}^{(3)}=\frac{dz}{z-w}-\frac{dz}{z}.
\end{align*}
We introduce the non-commutative polynomial algebras 
 $\A_z(\eta_2,\eta_{23}^{(2)},\eta_{22}):=\Q\langle \eta_2,\eta_{23}^{(2)},\eta_{22}\rangle$
(resp. $\A(\eta_3,\eta_{33}):=\Q\langle \eta_3,\eta_{33}\rangle$).
By the map $e_0\mapsto \eta_2$, $e_z\mapsto \eta_{23}^{(2)}$, $e_1\mapsto \eta_{22}$
(resp. $e_0\mapsto \eta_3$, $e_1\mapsto \eta_{33}$),
it is identified with $\A_z$  (resp. $\A$).
We set $\A^1_z(\eta_2,\eta_{23}^{(2)},\eta_{22})$, $\A^0_z(\eta_2,\eta_{23}^{(2)},\eta_{22})$
(resp. $\A^1(\eta_3,\eta_{33})$)
to be the subspaces corresponding to $\A_z^1$ and $\A_z^0$
(resp. $\A^1$)
under the identification.
Similarly we define corresponding subspaces for 
$\A_w(\eta_3,\eta_{23}^{(3)},\eta_{33}):=\Q\langle \eta_3,\eta_{23}^{(3)},\eta_{33}\rangle$
(resp. $\A(\eta_2,\eta_{22}):=\Q\langle\eta_2,\eta_{22}\rangle$)
by the identification with $\A_z$ (resp. $\A$)
given by 
$e_0\mapsto \eta_{3}$, $e_z\mapsto \eta_{23}^{(3)}$, $e_1\mapsto \eta_{33}$
(resp. $e_0\mapsto \eta_2$, $e_1\mapsto \eta_{22}$).

\begin{lem}\label{lem:dec2 and dec3}
There are 
the following isomorphisms of shuffle $\Q$-algebras
\begin{align*}
\dec_2:
=(\pr^{(2)}_{2\otimes 3}\otimes \pr^{(3)}_{2\otimes 3})\circ\delta: 
& \
\B\simeq
\A_z(\eta_2,\eta_{23}^{(2)},\eta_{22})\otimes_\Q\A(\eta_3,\eta_{33}), \\
\dec_3:
=(\pr^{(3)}_{3\otimes 2}\otimes \pr^{(2)}_{3\otimes 2})\circ\delta: 
& \
\B\simeq
\A_w(\eta_3,\eta_{23}^{(3)},\eta_{33})\otimes_\Q\A(\eta_2,\eta_{22}).
\end{align*}
\end{lem}

Here 
$\pr^{(2)}_{2\otimes 3}$ is the algebra homomorphism
sending $\eta_{3}, \eta_{33}\mapsto 0$, $\eta_{23}\mapsto\eta_{23}^{(2)}$
and
$\pr^{(3)}_{2\otimes 3}$ is the algebra homomorphism
sending $\eta_{2}, \eta_{23},\eta_{22}\mapsto 0$. 
Similarly the algebra homomorphisms
$\pr^{(3)}_{3\otimes 2}$ (resp. $\pr^{(2)}_{3\otimes 2}$) is defined by
$\eta_{2}, \eta_{22}\mapsto 0$, $\eta_{23}\mapsto\eta_{23}^{(3)}$
(resp. $\eta_{3}, \eta_{23}, \eta_{33}\mapsto 0$).

\begin{proof}
This is nothing but a reformulation of \cite{OU} Proposition 9.3, where
they employ the different coordinate $(z_1,z_2)$ corresponding  to our coordinate $(w,\frac{1}{z})$.
Actually their terminologies
are translated to ours as follows:
$$
\xi_1=\eta_2, \
\xi_{11}=-\eta_{22}, \
\xi_2=\eta_3, \
\xi_{22}=-\eta_{33}, \ 
\xi_{12}=-\eta_{23}
$$
and $\iota_{1\otimes 2}=\dec_2$, $\iota_{2\otimes 1}=\dec_3$
\end{proof}

We associate the above two decompositions with the following inclusions and projections of shuffle algebras
\begin{align*}
i_2:\A\hookrightarrow \B, \qquad
j_2:\A_z\hookrightarrow \B, \qquad   & r_2:\B\twoheadrightarrow \A_z, 
\\
i_3:\A\hookrightarrow \B, \qquad
j_3:\A_w\hookrightarrow \B, \qquad   & r_3:\B\twoheadrightarrow \A_w.
\end{align*}
In  precise,  they are defined by 
\begin{align*}
i_2&:\A\simeq \A(\eta_3,\eta_{33})\overset{1\otimes \id}{\hookrightarrow} 
\A_z(\eta_2,\eta_{23}^{(2)},\eta_{22})\otimes_\Q\A(\eta_3,\eta_{33})
\overset{\dec_2^{-1}}\simeq \B, \\
j_2&:\A_z\simeq \A_z(\eta_2,\eta_{23}^{(2)},\eta_{22})\overset{\id\otimes 1}{\hookrightarrow} 
\A_z(\eta_2,\eta_{23}^{(2)},\eta_{22})\otimes_\Q\A(\eta_3,\eta_{33})
\overset{\dec_2^{-1}}\simeq \B, \\
r_2&:\B\overset{\dec_2}\simeq 
\A_z(\eta_2,\eta_{23}^{(2)},\eta_{22})\otimes_\Q\A(\eta_3,\eta_{33})
\overset{\id\otimes\epsilon_\A}{\twoheadrightarrow} \A_z(\eta_2,\eta_{23}^{(2)},\eta_{22})
\simeq \A_z,
\end{align*}
where $\epsilon_\A$ is the augmentation map.
The maps  $i_3$, $j_3$ $r_3$ are defined similarly.
We have  $r_2\circ j_2=\id$ and $r_3\circ j_3=\id$.
We note that
$i_2$ (resp. $i_3$) is the Hopf algebra homomorphism
induced by $\pr_2\circ (3,4)$ (resp. $\pr_3$)
and 
$r_2$ (resp. $r_3$) is the Hopf algebra homomorphism
induced by the embedding $j_\X:\X(z)\hookrightarrow \M_{(0,5)}$
(resp. $j_\Y:\Y(w)\hookrightarrow \M_{(0,5)}$).
By our construction, any element $b\in\B$ is decomposed as
$b=\sum_ki_2(a_k)\sha j_2(a'_k)$ with $a_k\in\A$ and $a'_k\in\A_z$.

We denote $U\frak f_z=\Q\langle f_0,f_z,f_1\rangle$ to be the universal enveloping algebra 
of the free Lie algebra  $\frak f_z$ 
of three variables $f_0, f_z, f_1$ and regard it 
as the linear dual of $\A_z$ by the pairing $<\cdot|\cdot>$.
Since $r_2$ is induced  by the embedding $\X(z)\hookrightarrow \M_{(0,5)}$,
we have
\begin{equation}\label{eqn:on r2}
<r_2(l')|\varphi'>=<l'|r_2(\varphi')>
\end{equation}
for any $l'\in\B$ and $\varphi'\in U\frak f_3$.
Here $r_2$ on the right hand side means
the induced map $U\frak f_z\hookrightarrow U\frak P_5$
sending $f_0\mapsto  t_{21}$, $f_z\mapsto t_{23}$, $f_1\mapsto t_{24}$. 
Similarly we have
\begin{equation}\label{eqn:on r3}
<r_3(l')|\varphi'>=<l'|r_3(\varphi')>
\end{equation}
for any $l'\in\B$ and $\varphi'\in U\frak f_w:=\Q\langle f_0,f_w,f_1\rangle$.
Here $r_3$ on the right hand side means
the induced map $U\frak f_w\hookrightarrow U\frak P_5$
sending $f_0\mapsto  -t_{31}$, $f_w\mapsto -t_{31}+t_{23}$, $f_1\mapsto -t_{31}+t_{34}$. 

For $k=1,\dots,5$,
we denote  $\pr_k:\widehat{U{\mathfrak P}_5}\to \widehat{U{\mathfrak f}_2}$
to be the projection  induced from $\pr_k:\M_{0,5}\to \M_{0,4}$,
which actually sends $t_{ij}$ to $0$ when $i=k$ or $j=k$. 
Since $i_2$ and $i_3$ are induced from 
$\pr_2\circ (3,4)$ and $\pr_3$ respectively,
we have
\begin{equation}\label{eqn:on i2}
<i_2(l)|\varphi'>=<l \ | \ \pr_2\circ (3,4) (\varphi')>, \quad
<i_3(l)|\varphi'>=<l\ | \ \pr_3(\varphi')>
\end{equation} 
for any $l\in\A$ and $\varphi'\in \widehat{U{\mathfrak P}_5}$.
For our later use, we also consider the inclusion 
$i_4:\A\to \B$ of Hopf algebras sending
$e_0\mapsto e_{21}-e_{31}$ and
$e_1\mapsto e_{23}-e_{31}$,
which is induced from the projection
$\pr_4:\M_{0,5}\to\M_{0,4}$.
It induces an identification of pairings
\begin{equation}\label{eqn:on i4}
<i_4(l)|\varphi'>=<l\ | \ \pr_4(\varphi')>
\end{equation}

We put $\B^1$ 
(N.B. it was denoted by $\B^0$ in \cite{OU})
to be the subalgebra of $\B$
generated by elements which have no terms ending with $\eta_2$ and $\eta_3$.

\begin{lem}
We have the following decompositions:
\begin{align}
\label{eqn:dec2}
&\dec_2(\B^1)=
\A^1_z(\eta_2,\eta_{23}^{(2)},\eta_{22})\otimes_\Q\A^1(\eta_3,\eta_{33}),\\
\label{eqn:dec3}
&\dec_3(\B^1)=
\A^1_w(\eta_3,\eta_{23}^{(3)},\eta_{33})\otimes_\Q\A^1(\eta_2,\eta_{22}).
\end{align}
\end{lem}

\begin{proof}
It is stated in  \cite{OU} Proposition 9.4.
We get the claim by restriction of the isomorphisms of Lemma \ref{lem:dec2 and dec3}
to $\B^1$.
\end{proof}
It follows
\begin{equation}\label{B0A0}
\B^1\cap j_2(\A_z)= j_2(\A_z^1)\supset j_2(\A_z^0)
\end{equation}
and also
$\B^1\cap j_3(\A_w)= j_3(\A_w^1)\supset j_3(\A_w^0)$.

\subsection{Differentials and multiple polylogarithms}
We introduce the linear operator $\partial_\alpha$ ($\alpha=0,1$) on $j_3(A_w)$
which extends to the $i_3(\A)$-linear operator  $\tilde\partial_\alpha$ on $\B$
and show that it restricts to 
Hirose-Sato's differential operator
$\partial_{z,\alpha}$
on $j_2(\A_z)$ 
by showing how they are connected to multiple polylogarithms.

Let $\mathcal P_{(\infty,0)}^{(z,w)}(\M_{0,5})$ to the set of piece-wise smooth paths
from $(\infty,0)$ to $(z,w)$ on $\M_{0,5}$.
We consider the $\Q$-linear map
$$\tilde\rho_5:T\B(\M_{0,5})_1\to
\mathrm{Map}(\mathcal P_{(\infty,0)}^{(z,w)}(\M_{0,5}), \C).
$$
defined by regularized iterated integral with tangential basepoints
(for its treatment see also \cite{De}\S 15).
Particularly it is given by
\begin{equation*}\label{iterated integral map}
\tilde\rho_5([\omega_{i_m}|\cdots|\omega_{i_1}])(\gamma)=
\int_{0<t_1< \cdots <t_m<1}
\omega_{i_m}({\gamma(t_m)})\cdot
\omega_{i_{m-1}}({\gamma(t_{m-1})})\cdot\cdots
\omega_{i_1}({\gamma(t_1)})
\end{equation*}
when the integral is convergent.
We define 
$
I^{(z,w)}_{(\infty,0)}({\mathcal M_{0,5}})
$
to be 
the homotopy invariant part of $\mathrm{Im}\ \tilde\rho_5$.
We often regard  it to be a subspace of 
$\mathrm{Map}\left(\pi_1\left(\M_{0,5};(\infty,0),(z,w)\right), \C\right)$.
Then by  Chen's theory \cite{C},
we have a shuffle (actually Hopf) algebra isomorphism 
$$\rho_5:H^0\bar B(\Omega^\ast_\mathrm{DR}(\M_{0,5}))\simeq I^{(z,w)}_{(\infty,0)}({\mathcal M_{0,5}}).
$$
with $\rho_5=\tilde\rho_5|_{H^0\bar B(\Omega^\ast_\mathrm{DR}(\M_{0,5}))}$.
Similarly we have 
isomorphisms
\begin{align*}
&\rho_z:H^0\bar B(\Omega^\ast_\mathrm{DR}(\X(z)))\simeq I^{w}_{0}(\X(z)),\\
&\rho_w:H^0\bar B(\Omega^\ast_\mathrm{DR}(\Y(w)))\simeq I^{z}_{\infty}(\Y(w)).
\end{align*}

By the isomorphisms in Lemma \ref{lem:bar identifications},
we identify $\rho_5$, $\rho_z$, $\rho_w$ with the isomorphisms
$\B\simeq
I^{(z,w)}_{(\infty,0)}({\mathcal M_{0,5}})$,
$\A_z \simeq I^{w}_{0}(\X(z)) $,
$\A_w \simeq I^{z}_{\infty}(\Y(w)) $,
and use the same symbols. 

\begin{rem}\label{rem:rho and MPL}
(i).
We note that for each $l\in\A^0_z$
$$\rho_z(l)(\mathrm{dch}_{0,1})=L(l)
$$
where $\mathrm{dch}_{0,1}$ is
the straight line path  from $0$ to $1$ when $w=1$ and 
$L(l)$ is in \eqref{eqn:L}.

(ii).
For a multi-index ${\bf k}=(k_1,\dots,k_m)\in\N^{m}$, 
set ${\bf e}_{\bf k}(1)=e_0^{k_1-1}e_1\cdots e_0^{k_m-1}e_1$ and
${\bf e}_{\bf k}(z)=e_0^{k_1-1}e_z\cdots e_0^{k_m-1}e_z$.
Then the set 
$$
\{{\bf e}_{{\bf k}_1}(1){\bf e}_{{\bf k}_2}(z){\bf e}_{{\bf k}_3}(1)
\cdots {\bf e}_{{\bf k}_N}(\ast) \mid
N>0,m_1\geqslant 0, m_2,\dots,m_N>0,{\bf k}_i\in\N^{m_i}
\}
$$ ($\ast$ is $1$ or $z$ according to the parity of $N$)
forms a basis of $\A^1_z$. 
For such $l={\bf e}_{{\bf k}_1}(1){\bf e}_{{\bf k}_2}(z){\bf e}_{{\bf k}_3}(1)
\cdots {\bf e}_{{\bf k}_N}(\ast)$,
its image under $\rho_z$
is calculated to be 
$$
\rho_z(l)=
(-1)^{\sum_{i=1}^Nm_i}\cdot \ell_{{\bf k}_1,{\bf k}_2,\dots,{\bf k}_N}.
$$
Here $\ell_{{\bf k}_1,{\bf k}_2,\dots,{\bf k}_N}$
is the element in $I^{w}_{0}(\X(z))$
which is associated with the series
\begin{equation}\label{eqn:rho=Li}
\Li_{{\bf k}_1,{\bf k}_2,\dots,{\bf k}_N}
(w,{\bf 1}^{m_1-1},z^{-1},{\bf 1}^{m_2-1},z,{\bf 1}^{m_3-1},z^{-1},\dots,{\bf 1}^{m_N-1})
\end{equation}
(when $m_1=0$,  it stands for
$
\Li_{{\bf k}_2,\dots,{\bf k}_N}
({w}{z}^{-1},{\bf 1}^{m_2-1},z,{\bf 1}^{m_3-1},z^{-1},\dots,{\bf 1}^{m_N-1})
$),
that is,
the restriction of the {\it multiple polylogarithm}
$$
\Li_{k_1,\dots,k_m}(s_1,\dots,s_m):=\sum_{n_1>\cdots>n_m>0}
\frac{s_1^{n_1}\cdots s_m^{n_m}}{n_1^{k_1}\cdots n_m^{k_m}}
$$
to $(s_1,\dots,s_m)=(w,{\bf 1}^{m_1-1},z^{-1},{\bf 1}^{m_2-1},z,{\bf 1}^{m_3-1},z^{-1},\dots,{\bf 1}^{m_N-1})$
where $(k_1,\dots,k_m)$ is the juxtaposition of 
${\bf k}_1,{\bf k}_2,\dots,{\bf k}_N$.
Strictly speaking  \eqref{eqn:rho=Li} is 
merely a series converging
when $|w|<|z|$, however it determines an element of 
$I^{w}_{0}(\X(z))$ 
by its iterated integral presentation  
even if $|w|\geqslant |z|$.

(iii).
By abuse of notation, we denote 
$j_2: I^{w}_{0}(\X(z))\hookrightarrow I^{(z,w)}_{(\infty,0)}({\M_{0,5}})$
the embedding  induced from $j_2$ under the two isomorphisms $\rho_z$ and $\rho_5$:
\begin{equation}\label{eq:CD:rho and j}
\xymatrix{
\B\ar^-{\rho_5}@{->}[r] & I^{(z,w)}_{(\infty,0)}({\M_{0,5}}) \\
\A_z\ar^{j_2}@{^{(}->}[u]\ar^-{\rho_z}[r] &  I^{w}_{0}(\X(z))\ar^{j_2}@{^{(}->}[u]
}
\end{equation}
By our construction,
the image $j_2(\ell_{{\bf k}_1,{\bf k}_2,\dots,{\bf k}_N})$
is nothing but the element 
$\ell_{{\bf k}_1,{\bf k}_2,\dots,{\bf k}_N}^{w,z^{-1}}$
in
$I^{(z,w)}_{(\infty,0)}({\M_{0,5}})$
which is associated with the series 
regarding both $w$ and $z$ as variables in \eqref{eqn:rho=Li}.
Recursive differentiations of \eqref{eqn:rho=Li} with respect to $w$ and $z$
reveal the expression of $\rho_5^{-1}(j_2(\ell_{{\bf k}_1,{\bf k}_2,\dots,{\bf k}_N}))$
in $\B$.
The differentiations assure 
$\dec_2(\rho_5^{-1}(\ell_{{\bf k}_1,{\bf k}_2,\dots,{\bf k}_N}^{w,z^{-1}}))
=\rho_z^{-1}(\ell_{{\bf k}_1,{\bf k}_2,\dots,{\bf k}_N})\otimes 1$
because each term of 
$\rho_5^{-1}(\ell_{{\bf k}_1,{\bf k}_2,\dots,{\bf k}_N}^{w,z^{-1}})$
never ends in $\eta_3$ or $\eta_{33}$.

(iv).
Similarly we have an element 
$\ell_{{\bf k}_1,{\bf k}_2,\dots,{\bf k}_N}^{z^{-1},w}$
in
$I^{(z,w)}_{(\infty,0)}({\M_{0,5}})$
which is defined by 
$$
\Li_{{\bf k}_1,{\bf k}_2,\dots,{\bf k}_N}
(z^{-1},{\bf 1}^{m_1-1},w,{\bf 1}^{m_2-1},w^{-1},{\bf 1}^{m_3-1},w,\dots,{\bf 1}^{m_N-1})
$$
and is also equal to 
$j_3\circ\rho_w((-1)^{\sum_{i=1}^Nm_i}\cdot 
{\bf e}_{{\bf k}_1}(1){\bf e}_{{\bf k}_2}(w){\bf e}_{{\bf k}_3}(1)
\cdots {\bf e}_{{\bf k}_N}(\ast)
)$.
\end{rem}


We consider new operators $\partial_{\alpha}$ on $\A_w$
and $\tilde\partial_{\alpha}$ on $\B$ ($\alpha=0,1$) below:

\begin{defn}
(1).
For each $l\in\A_w$ with the presentation 
$l=[e_0|l_1]+[e_w|l_2]+[e_1|l_3]$ with $l_1,l_2,l_3\in\A_w$,
we define
$$
\partial_{0}(l)=-l_1-l_2-l_3\in\A_w,  \qquad 
\partial_{1}(l)=l_2+l_3\in\A_w,
$$
and
$\partial_{0}(1)=\partial_{1}(1)=0\in\A_w$,
which yield linear operators
$\partial_{0}$ and $\partial_{1}$ on $\A_w$.

(2). 
By using $\dec_3$,
we extend  $\partial_{\alpha}$
to $i_3(\A)$-linear (under the shuffle product)
operator $\tilde\partial_{\alpha}$ on $\B$ for $\alpha=0,1$:
$$
\tilde\partial_{\alpha}(j_3(l)\sha i_3(m))
=j_3(\partial_{\alpha}(l))\sha i_3(m)
$$
for $l\in\A_w$ and $m\in\A$.
\end{defn}

By Remark \ref{rem:rho and MPL} (iv),
the image of $l\in\B$ with the presentation
$l=[e_{31}|l_1]+[e_{34}|l_2]+[e_{23}|l_3]+[e_{21}|l_4]+[e_{24}|l_5]$
($l_1,\dots,l_5\in\B$)
is calculated to be
\begin{equation}\label{eq:tilde partial01}
\tilde\partial_{0}(l)=l_1\in\B,  \qquad 
\tilde\partial_{1}(l)=l_2+l_3\in\B.
\end{equation}


The following is a key proposition which says that our differential operator
$\tilde\partial_{\alpha}$
extends the one $\partial_{z,\alpha}$ of Sato-Hirose.
\begin{prop}\label{prop: two partials}
The operators
$\tilde\partial_{0}$ and $\tilde\partial_{1}$ on $\B$
restrict to the linear operators $\partial_{z,0}$ and $\partial_{z,1}$ on $\A_z$, that is, 
we have the following commutative diagram for $\alpha=0,1$:
$$
\xymatrix{
\B\ar^{\tilde\partial_{\alpha}}[r] &\B \\
\A_z\ar^{j_2}@{^{(}->}[u]\ar^{\partial_{z,\alpha}}[r] &\A_z\ar^{j_2}@{^{(}->}[u]
}
$$
\end{prop}

\begin{proof}
By the same arguments of the proof of \cite{Gon} Theorem 2.1
(for admissible indices), we have
\begin{align*}
dI(a_0;&a_1, \dots, a_m;a_{m+1}) \\
&=\sum_{i=1}^m
I(a_0;a_1,\dots, \Check{a}_i,\dots;a_{m+1})\cdot \{d\log (a_i-a_{i+1})-d\log (a_i-a_{i-1})\}
\end{align*}
for 
$I(a_0;a_1,\dots, a_m;a_{m+1})=\int_{a_0}^{a_{m+1}}
\frac{dt_m}{t_m-a_m}\wedge \cdots \wedge \frac{dt_1}{t_1-a_1}$.
Then we have
\begin{equation}\label{eq:rho z l'}
\frac{d}{dz}\rho_z(l')(\gamma')
=\frac{1}{z}\rho_z(\partial'_{z,0}(l'))(\gamma')
+\frac{1}{z-1}\rho_z(\partial'_{z,1}(l'))(\gamma')
+\frac{1}{z-w}\rho_z(\partial'_{z,w}(l'))(\gamma')
\end{equation}
for any $l'\in \A_z$
and $\gamma'\in \pi_1(\X(z),0,w)$, 
where $\partial'_{z,\alpha}$ ($\alpha\in\{0,1,z,w\}$) is the linear operator defined by
$$
\partial'_{z,\alpha}(e_{a_n}\cdots e_{a_1})=
\sum\nolimits_{i=1}^n\left(\delta_{\{a_i,a_{i+1}\},\{z,\alpha\}}-\delta_{\{a_{i-1},a_{i}\},\{z,\alpha\}}\right)e_{a_n}\cdots \Check{e}_{a_i}\cdots e_{a_1}
$$
with $a_0=0$ and $a_{n+1}=w$.
By comparing this with \eqref{eqn:partial z alpha}, we get
\begin{equation}\label{eqn:two partials} 
\partial_{z,0}=\partial'_{z,0} \text{  and  }
\partial_{z,1}=\partial'_{z,1}+\partial'_{z,w}.
\end{equation}

While, by definition, for $l\in\B$ with 
$l=[e_{31}|l_1]+[e_{34}|l_2]+[e_{23}|l_3]+[e_{21}|l_4]+[e_{24}|l_5]$
($l_1,\dots,l_5\in\B$),
we have
\begin{align*}
d \rho_5(l)(\gamma)
=\frac{d  z}{z}\rho_5(l_1)(\gamma)
&+\frac{d  z}{z-1}\rho_5(l_2)(\gamma)
+\frac{d  w-d  z}{w-z}\rho_5(l_3)(\gamma) \\
&+\frac{d  w}{w}\rho_5(l_4)(\gamma)
+\frac{d  w}{w-1}\rho_5(l_5)(\gamma)
\end{align*}
for any  $\gamma\in \pi_1({\mathcal M_{0,5}};{(\infty,0)}, (z,w))$.
Particularly we have
$$
\frac{d}{dz}\rho_5(l)(\gamma)
=\frac{1}{z}\rho_5(l_1)(\gamma)
+\frac{1}{z-1}\rho_5(l_2)(\gamma)
+\frac{1}{z-w}\rho_5(l_3)(\gamma).
$$
Hence 
for 
$L\in\B$
such that
\begin{equation}\label{eqn:differential2}
\frac{d}{dz}\rho_5(L)(\gamma)
=\frac{1}{z}\rho_5(L_1)(\gamma)
+\frac{1}{z-1}\rho_5(L_2)(\gamma)
+\frac{1}{z-w}\rho_5(L_3)(\gamma)
\end{equation}
with 
$L_i\in\B$
($i=1,2,3$),
we have
\begin{equation}\label{eqn:differential1}
\tilde\partial_{0}(L)=L_1  \quad\text{   and     }\quad
\tilde\partial_{1}(L)=L_2+L_3
\end{equation}
by \eqref{eq:tilde partial01}.

By Remark \ref{rem:rho and MPL}. (iii),
the restrictions of $\tilde\partial_{0}$ and $\tilde\partial_{1}$
to $I_{0}^w(\X(z))=\mathrm{Im} \ \rho_z$ by the inclusion $j_2$ 
uner the identifications given in \eqref{eq:CD:rho and j} 
are also calculated  by \eqref{eqn:differential2} and \eqref{eqn:differential1}.
Hence by \eqref{eq:rho z l'}
we have
\begin{equation}\label{eq:tilde partial 01=j2 partial'}
\tilde\partial_{0}(j_2(l'))=j_2(\partial'_{z,0}(l')), \quad 
\tilde\partial_{1}(j_2(l'))=j_2(\partial'_{z,1}(l')+\partial'_{z,w}(l'))
\end{equation}
for $l'\in\A_z$.

By \eqref{eqn:two partials} and \eqref{eq:tilde partial 01=j2 partial'},
we have
\begin{align*}
\tilde\partial_{0}(j_2(l'))
&
=j_2(\partial'_{z,0}(l'))  
=j_2(\partial_{z,0}(l'))  \\
\tilde\partial_{1}(j_2(l'))
&
=j_2(\partial'_{z,1}(l')+\partial'_{z,w}(l'))
=j_2(\partial_{z,1}(l')).
\end{align*}
Therefore our claim follows. 
\end{proof}

\subsection{Upgrading of the standard relations and the involution $\tau_z$}
We extend the set of  standard relations to a subspace of $\B$
and give its explicit presentation.
We also extend the involution $\tau_z$ in $\A_z$ to the one in $\B$
and show its property.

We start by preparing the $\Q$-linear subspace of $\B$  
\begin{equation}\label{eq:LST}
\tilde{\mathcal L}_{\mathrm{ST}}:=
\langle j_3(l-l_{1\to w})\sha i_3(m) \mid
l\in \A^1_w,\  m\in\A^1
\rangle_\Q.
\end{equation}
Here $l_{1\to w}$ means the image of $l$ by the algebra homomorphism
$\A_w\to \A_w$ sending $e_0\mapsto e_0$, $e_w\mapsto e_w$, $e_1\mapsto e_w$.

\begin{rem}\label{rem:MPL presentation of B and LST}
By the isomorphism $\rho_5:\B\simeq I^{(z,w)}_{(\infty,0)}({\mathcal M_{0,5}})$ 
and the decompositions $\dec_2$ and $\dec_3$,
the $\Q$-linear subspaces
$\B^1$ and $\tilde{\mathcal L}_{\mathrm{ST}}$  are described in terms of 
the elements of Remark \ref{rem:rho and MPL}:

(i).
By \eqref{eqn:dec2}, we see that
$\rho_5(\B^1)$ is linearly spanned by the basis 
$\ell_{{\bf k}_1,{\bf k}_2,\dots,{\bf k}_N}^{w,z^{-1}}\cdot \ell_{\bf h}^{z^{-1}}$
(where $\ell_{\bf h}^{z^{-1}}$ is the element 
in $I^{(z,w)}_{(\infty,0)}({\mathcal M_{0,5}})$ associated with 
$\Li_{\bf h}(z^{-1},1,1,\dots)$),
which is an element of $I^{(z,w)}_{(\infty,0)}({\mathcal M_{0,5}})$ determined by 
$$
\Li_{{\bf k}_1,{\bf k}_2,\dots,{\bf k}_N
}(w,{\bf 1}^{m_1-1},z^{-1},{\bf 1}^{m_2-1},z,\dots)
\cdot \Li_{\bf h}(z^{-1},1,1,\dots),
$$
for all multi-indices ${\bf h}$, ${\bf k}_1$,\dots, ${\bf k}_N$,
$N>0$ (${\bf h}$ and ${\bf k}_1$ can be empty).\\
While by \eqref{eqn:dec3}, similarly we  see that
$\rho_5(\B^1)$  is linearly spanned by the basis 
$\ell_{{\bf k}_1,{\bf k}_2,\dots,{\bf k}_N}^{z^{-1},w}\cdot \ell_{\bf h}^w$,
which is an element of $I^{(z,w)}_{(\infty,0)}({\mathcal M_{0,5}})$ determined by  
$$
\Li_{{\bf k}_1,{\bf k}_2,\dots,{\bf k}_N
}(z^{-1},{\bf 1}^{m_1-1},w,{\bf 1}^{m_2-1},w^{-1},\dots)
\cdot \Li_{\bf h}({w},1,1,\dots),
$$
for all multi-indices ${\bf h}$, ${\bf k}_1$,\dots, ${\bf k}_N$,
$N>0$ (${\bf h}$ and ${\bf k}_1$ can be empty). 

(ii).
Particularly
$\rho_5(\tilde{\mathcal L}_{\mathrm{ST}})$
is linearly spanned by the basis 
\begin{equation*}
\{\ell_{{\bf k}_1,{\bf k}_2,\dots,{\bf k}_N}^{z^{-1},w}-
\ell_{\emptyset,({\bf k}_1{\bf k}_2\dots{\bf k}_N)}^{z^{-1},w}\}
\cdot \ell_{\bf h}^w,
\end{equation*}
that is, the element of $I^{(z,w)}_{(\infty,0)}({\mathcal M_{0,5}})$
determined by 
\begin{align*}
&\left\{\Li_{{\bf k}_1,\dots,{\bf k}_N
}(z^{-1},{\bf 1}^{m_1-1},w,{\bf 1}^{m_2-1},w^{-1},\dots)-
\Li_{{\bf k}_1,\dots,{\bf k}_N
}(wz^{-1},{\bf 1}^{m_1+\cdots+m_N-1})
\right\} \\
&\qquad\quad
\cdot \Li_{\bf h}({w},1,1,\dots),
\end{align*}
for all multi-indices ${\bf h}$, ${\bf k}_1$,\dots, ${\bf k}_N$,
$N>0$ (${\bf h}$ and ${\bf k}_1$ can be empty). 
\end{rem}

We extend the set of standard relations to a subspace of $\B$.

\begin{defn}
(i).
Let $\epsilon:\A_w\to \Q$ be the augmentation map.
By using $\dec_3$, we extend $\epsilon$
to $i_3(\A)$-linear map
$\tilde\epsilon:\B \to \A$, that is,
$$
\tilde\epsilon(j_3(l)\sha i_3(m))
=j_3(\epsilon(l)) \sha i_3(m)
=\epsilon(l) i_3(m)
$$
for $l\in\A_w$ and $m\in\A$.

(ii).
We define the $\Q$-linear subspace of $\B$  
$$
\tilde{\mathcal I}_{\mathrm{ST}}:=\{l\in\B^1\ \mid 
\tilde\epsilon \ \tilde\partial_{\alpha_1}\cdots\tilde\partial_{\alpha_r}(l)
=0
\text{ for } r\geqslant 0,  \alpha_1,\dots,\alpha_r\in\{0,1\}
\}.
$$
\end{defn}

The map $\tilde\epsilon:\B\to\A$ is calculated to be
the shuffle algebra homomorphism
sending 
$$
e_{31}\mapsto 0, \quad
e_{23}\mapsto 0, \quad 
e_{34}\mapsto 0, \quad
e_{21}\mapsto e_{0}, \quad 
e_{24}\mapsto e_{1},
$$
which yileds
the following commutative  diagram: 
\begin{equation}\label{diag:Const}
\xymatrix{
\B\ar^{\tilde\epsilon}[r] &\A \\
\A_z\ar^{j_2}@{^{(}->}[u]\ar^{\Const}[r] &\A\ar@{=}[u]
}
\end{equation}

Then by \eqref{B0A0}, \eqref{diag:Const} and Proposition \ref{prop: two partials} we have
\begin{equation}\label{ISTA0}
\tilde{\mathcal I}_{\mathrm{ST}}\cap j_2(\A^0_z)=j_2({\mathcal I}_{\mathrm{ST}}).
\end{equation}
Actually $\tilde{\mathcal I}_{\mathrm{ST}}$ coincides with
our previous space $\tilde{\mathcal L}_{\mathrm{ST}}$ in \eqref{eq:LST}.

\begin{prop}\label{basis presentation}
We have 
$$
\tilde{\mathcal I}_{\mathrm{ST}}=
\tilde{\mathcal L}_{\mathrm{ST}}.
$$
\end{prop}

\begin{proof}
Since $i_3(\A)$-linear maps
$\tilde\epsilon$ and $\tilde\partial_\alpha$ are constructed 
as scalar extensions of $\Q$-linear maps $\epsilon$ and $\partial_\alpha$
and  we have $\B^1= j_3(\A^1_w)\sha i_3(\A^1)$,
we have
$$
\tilde\I_\ST=\{\tilde\I_\ST\cap j_3(\A^1_w)\}\sha i_3(\A^1).
$$
While we have
$$
\tilde{\mathcal L}_\ST=
\langle j_3(l-l_{1\to w}) \mid
l\in \A^1_w\rangle_\Q \sha
i_3(\A^1).
$$
It is enough to show 
$$\I_w={\mathcal L}_w$$ 
with
$
\I_w:=
\{l\in \A_w^1 \mid \epsilon \ \partial_{\alpha_1}\cdots\partial_{\alpha_r}(l)
=0
\text{ for } r\geqslant 0,  \alpha_1,\dots,\alpha_r\in\{0,1\}
\}
$
and
${\mathcal L}_w:=\langle l-l_{1\to w} \mid l\in \A^1_w\rangle_\Q$.

Put $P=e_0$, $Q=e_1+e_w$, $R=e_1-e_w$ in $\A_w$
and express any  element $l\in \A_w$ as
$f(P,Q,R)$ with a $3$ variables non-commutative polynomial $f$. 
By definition, $\l\in {\mathcal I}_w$ if and only if 
$\partial_{0}(l)\in {\mathcal I}_w$, 
$\partial_{1}(l)\in{\mathcal I}_w$
and $\epsilon(l)=0$.
It is  reformulated to be 
$f_1(P,Q,R),f_2(P,Q,R)\in \I_w$ and $f(0,0,0)=0$,
when it is given by  $f(P,Q,R)=Pf_1(P,Q,R)+Qf_2(P,Q,R)+Rf_3(P,Q,R)+f(0,0,0)$.
Actually the conditions are equivalent to $f(P,Q,0)=0$,
from which we learn 
$\I_w={\mathcal L}_w$.
\end{proof}

There is an action of the symmetric group $\frak S_5$ on $\M_{0,5}$
by permutation of 5 marked points.
Particularly we consider the involution $\tau=(1,4)(3,5)\in\frak S_5$.
It induces the involution  $\tau$ on $U\frak P_5$ by $t_{i,j}\mapsto t_{\tau(i),\tau(j)}$  and hence on its dual space $\B$ so that
for any $l\in U\frak P_5$ and $\phi\in\B$
\begin{equation}\label{pairing equation}
<\tau(l) |\phi >=<l|\tau(\phi)>.
\end{equation}

\begin{lem}\label{lem:new involution}
Let $S$ be the antipode of $\A_z$. The following diagram is commutative
$$
\xymatrix{
\B\ar^{\tau}[r]\ar_{r_2}@{->>}[d] &\B \ar^{r_2}@{->>}[d]\\
\A_z\ar^{S\circ \tau_z}[r] &\A_z 
}
$$
\end{lem}

\begin{proof}
In our coordinate $(z,w)$, the involution $\tau$ is described as
$(z,w)\mapsto (z,\frac{z(w-1)}{w-z})$.
Hence the induced involution $\tau$ on $\B$ is presented as
$e_{21}\mapsto e_{31}+e_{24}-e_{23}$, \
$e_{23}\mapsto e_{31}+e_{34}-e_{23}$, \
$e_{24}\mapsto e_{21}+e_{34}-e_{23}$, \
$e_{31}\mapsto e_{31}$, \
$e_{34}\mapsto e_{34}$.
While the action $\tau$ on  $\M_{0,5}$ 
stabilizes the fiber $\pr_2^{-1}(z)=\X(z)$ for a fixed $z$ and
induces the M\"{o}bius transformation $t\mapsto \frac{z(t-1)}{t-z}$ there.
It causes the involution on $\A_z$ given by 
$e_0\mapsto e_1-e_z$, $e_z\mapsto -e_z$, $e_1\mapsto e_0-e_z$
(cf. \cite{HIST}).
It coincides with  the automorphism $S\circ \tau_z$ of $\A_z$.
Since the map $r_2$ is given by
$r_2(e_{21})=e_0$, \ 
$r_2(e_{23})=e_z$, \ 
$r_2(e_{24})=e_1$, \ 
$r_2(e_{31})=0$, \ 
$r_2(e_{34})=0$,
we see that the diagram commutes.
\end{proof}

\subsection{Computations of the pairing}
We prove a key formula on the pairing in Proposition \ref{prop: key formula}
and then prove  Theorem \ref{thm:main 1}.

Since the triple $(\B,\sha,\delta)$ is the dual of the Hopf algebra 
$(U\frak P_5,\cdot, \Delta)$ with the standard coproduct given by 
$\Delta(t_{ij})= t_{ij}\otimes 1 + 1\otimes t_{ij}$,
we have
\begin{equation}\label{frequent equation1}
<l_1\sha l_2| \psi>=\sum\nolimits_i <l_1|\psi_1^{(i)}>\cdot<l_2|\psi_2^{(i)}>
\end{equation}
for $l_1,l_2\in \B$ and $\psi\in U\frak P_5$ with $\Delta(\psi)=\sum_i \psi_1^{(i)}\otimes \psi_2^{(i)}$, and
\begin{equation}\label{frequent equation2}
<l|\psi_1\cdot \psi_2>=\sum\nolimits_i <l_1^{(i)}|\psi_1>\cdot<l_2^{(i)}|\psi_2>
\end{equation}
for $\psi_1,\psi_2\in U\frak P_5$ and
$l\in \B$ with $\delta(l)=\sum_i l_1^{(i)}\otimes l_2^{(i)}$.
The same equations hold for the pairing $<\cdot | \cdot>$ between
$\A$ and $U\frak f_2$.

\begin{prop}\label{prop: key formula}
For $l\in \A_z^0$ and
$\varphi\in \exp [\hat{\mathfrak f}_2,\hat{\mathfrak f}_2]$,
we have
$$
<\lambda(l)|\varphi>=<j_2(l)|\varphi^{-1}_{243}\varphi_{215}\varphi_{534}>.
$$
\end{prop}

\begin{proof}
Since $\lambda:\A^0_z\to\A^0$ is a shuffle algebra homomorphism,
$\A_z^0$  is isomorphic to $\A_z^{-2}\otimes (\A_z^0\cap \Q\langle e_1,e_z\rangle)$
as shuffle algebras
and $\varphi$ is group-like,
it is enough to check the equality for the case  (i)  $l\in\A_z^{-2}$ or 
(ii) $l\in\A_z^0\cap \Q\langle e_1,e_z\rangle$
by \eqref{frequent equation1}.

(i). It is enough to check for $l\in \A_z^{-2}$ determined by \eqref{eqn:rho=Li}.
By definition, we have $\lambda(l)=l|_{z\to 1}$. Hence
\begin{align*}
<j_2(l)|\varphi^{-1}_{243}\varphi_{215}\varphi_{534}>
&
=<j_2(l)|\varphi_{215}\varphi_{534}>
=<j_2(l)|\varphi_{215}> \\
&=<l|\varphi(f_0,f_1+f_z)>  
=<l_{z\to 1}\mid\varphi>=<\lambda(l)|\varphi>.
\end{align*}
Here 
\begin{itemize}
\item 
The first equality is by \eqref{eqn:rho=Li}, \eqref{frequent equation2} and
$\varphi(f_0,0)=\varphi(0,f_1)=1$
since $j_2(l)$ is calculated by the prescription in Remark \ref{rem:rho and MPL}.(iii);
Recursive differentiations of \eqref{eqn:rho=Li} with respect to $w$ and $z$
tell us
that there always appears $e_{21}$ or $e_{31}$ prior to $e_{24}$
in each term of $j_2(l)$.
Hence $\varphi_{243}$ never contribute to the pairing.
\item
The second one is again by  \eqref{eqn:rho=Li} and \eqref{frequent equation2};
$j_2(l)$ is of the form $[l'|e_{24}]+[l''|e_{23}-e_{31}]$ ($l', l''\in \B$).
So $\varphi_{534}$ never contribute to the pairing 
because
$\varphi(t_{53},t_{34})=\varphi(-t_{31}-t_{32}-t_{34},t_{34})$.
\item
The third one follows from 
$r_2\circ j_2=\id$, \eqref{eqn:on r2} and
$\varphi(t_{21},t_{15})=\varphi(t_{21},t_{23}+t_{24}+t_{34})
=\varphi(t_{21},t_{23}+t_{24})$ where we use 
$[t_{34},t_{21}]=[t_{34},t_{23}+t_{24}]=0$ in the second equality.
\item
The fourth and the last ones are immediate.
\end{itemize}

(ii). It is enough to check for $l\in\A_z^0\cap \Q\langle e_1,e_z\rangle$ determined by \eqref{eqn:rho=Li}
with $({\bf k}_1,\dots,{\bf k}_N)=({\bf 1}^{m_1},\dots,{\bf 1}^{m_N})$.
By definition, we have
\begin{align*}
<\lambda(l)|\varphi>&=<(\Const\otimes\id)\circ \reg_{z}\circ N(l)|\varphi(f_0,f_1+f_z)>
=<N(l)|\varphi(f_0,f_1+f_z)> \\
&
=<\tau_z(l)|\varphi(f_0,f_1+f_z)> 
=<S\circ\tau_z(l)|\varphi(f_0,f_1+f_z)^{-1}> \\
&=<r_2\circ \tau\circ j_2(l)|\varphi(f_0,f_1+f_z)^{-1}> 
=<\tau\circ j_2(l)|\varphi_{215}^{-1}> 
=<j_2(l)|\varphi_{243}^{-1}> \\
&=<j_2(l)|\varphi^{-1}_{243}\varphi_{215}>
=<j_2(l)|\varphi^{-1}_{243}\varphi_{215}\varphi_{534}>.
\end{align*}
Here
\begin{itemize}
\item 
The first equality is by definition.
\item
The second one follows from the fact that $\varphi$ is commutator group-like and
$N(l)\in\A_z^{-2}$.
\item
The third one  is due to $\tau_z(l)\in \Q\langle e_0,e_z\rangle$.
\item
The fourth one is since we have $S_U(\varphi)=\varphi^{-1}$ for a group-like series $\varphi$ where $S_U$ is the antipode of $U{\mathfrak f}_2$.
\item
The fifth one is by $r_2\circ j_2=\id$ and  Lemma \ref{lem:new involution}.
\item
The sixth one is due to $\varphi(t_{21},t_{15})=\varphi(t_{21},t_{23}+t_{24})$
and \eqref{eqn:on r2}.
\item
The seventh one is by \eqref{pairing equation}.
\item
The eighth one is by \eqref{frequent equation2} and
$\varphi(t_{21},t_{15})=\varphi(t_{21},t_{23}+t_{24})$;
Since $e_{21}$ never appear in $j_2(l)$, $\varphi_{215}$ never contribute to the pairing.
\item
The last one is due to
the same reason to the second one of  our previous case (i). 
\end{itemize}
Therefore we get the claim.
\end{proof}

%

The following is also required to the proof of Theorem \ref{thm:main 1}.

\begin{lem}
For any $\ell\in\tilde{\mathcal I}_{\mathrm{ST}}$,
we have
\begin{equation}\label{eq: 351 124}
<\ell|\varphi_{351}\varphi_{124}>=0.
\end{equation}
\end{lem}

\begin{proof}
For any $\l\in\A^1_w$
we have
\begin{align}\label{eq:j3l 35124}
<j_3(l)|\varphi_{351}\varphi_{124}>&=<j_3(l)|\varphi_{351}> \\
&=<l|\varphi(3f_0-f_w-f_1,-2f_0+f_w+f_1)>. \notag
\end{align}
Here
\begin{itemize}
\item
The first equality is
by the same reason to the second equality in the case (i)
of the proof of Proposition \ref{prop: key formula};
Since $j_3(l)$ is of the form $[l'|e_{31}-e_{34}]+[l''|e_{23}-e_{31}]$ ($l', l''\in \B$),
$\varphi_{124}$ never contribute to the pairing.
\item
The second one follows from 
$r_3\circ j_3=\id$, \eqref{eqn:on r3} and
$\varphi_{351}=\varphi(t_{35},t_{51})
=\varphi(-t_{31}-t_{23}-t_{34},t_{23}+t_{34})$
where we use 
$[t_{24},t_{35}]=[t_{24},t_{23}+t_{34}]=0$ in the last equality.
\end{itemize}
It is immediate to see
$$
<l|\varphi(3f_0-f_w-f_1,-2f_0+f_w+f_1)>=
<l_{1\to w}|\varphi(3f_0-f_w-f_1,-2f_0+f_w+f_1)>
$$
for any $l\in \A_w$. Therefore
we have
\begin{equation}\label{eqn:j3 315}
<j_3(l)|\varphi_{351}>=<j_3(l_{1\to w})|\varphi_{351}>,
\end{equation}
whence
$
<j_3(l)|\varphi_{351}\varphi_{124}>=<j_3(l_{1\to w})|\varphi_{351}\varphi_{124}>
$
for any $l\in \A^1_w$.
Thus we have
\begin{align*}
<j_3(l-l_{1\to w}) & \sha i_3(m)|\varphi_{351}\varphi_{124}> \\
&=<j_3(l-l_{1\to w})|\varphi_{351}\varphi_{124}>\cdot
<i_3(m)|\varphi_{351}\varphi_{124}> 
=0
\end{align*}
for any $l\in\A^1_w$ and $m\in\A^1$,
which proves our claim by  Proposition \ref{basis presentation}.
\end{proof}

The proof of Theorem \ref{thm:main 1} goes as follows:
Assume that  
$\varphi$ is a commutator group-like series satisfying the pentagon equation.
By \cite{F10}, we have the 2-cycle relation
$\varphi(f_0,f_1)\varphi(f_1,f_0)=1$.
Then we have
$\varphi_{215}\varphi_{534}=\varphi_{243}\varphi_{351}\varphi_{124}$.
By Proposition \ref{prop: key formula}, we have for any $l\in\A_z^0$
$$
<\lambda(l)|\varphi>=<j_2(l)|\varphi^{-1}_{243}\varphi_{215}\varphi_{534}>
=<j_2(l)|\varphi_{351}\varphi_{124}>.
$$
By \eqref{ISTA0} and \eqref{eq: 351 124},
we have $<j_2(l)|\varphi_{351}\varphi_{124}>=0$ for any $l\in{\mathcal I}_{\mathrm{ST}}\subset \A_z^0$,
which proves $<\lambda(l)|\varphi>=0$ for 
any $l\in{\mathcal I}_{\mathrm{ST}}$.
\qed

\section{The confluence relations imply the  pentagon equation}
We prove that the confluence relations imply the pentagon equation under the commutator group-like condition
(Theorem \ref{thm:main 2}).

\begin{thm}\label{thm:main 2}
Let 
$\varphi\in \exp [\hat{\mathfrak f}_2,\hat{\mathfrak f}_2].$
If it satisfies the confluence relation, i.e.
$<l \mid\varphi>=0$ for any $l\in{\mathcal I}_{\CF}$,
then it satisfies the pentagon equation.
\end{thm}

\subsection{Construction of a standard relation}
We show how to construct a standard relation with a given element $l\in\A^0_z$.
\begin{lem}
For each $l\in\A^0_z$,
decompose $j_2(\l)=\sum_i j_3(l^{(i)})\sha i_3(m^{(i)})$
with $l^{(i)}\in \A^1_w$ 
and $m^{(i)}\in \A^1$ 
by $\dec_3$.
Then
$\sum_i j_3(l^{(i)}_{1\to w})\sha i_3(m^{(i)})$ belongs to
$j_2(\A_z^0)$.
\end{lem}

\begin{proof}
We may assume  that $l$ corresponds to
$\ell_{{\bf k}_1,{\bf k}_2,\dots,{\bf k}_N}\in I^{w}_{0}(\X(z))$ under the isomorphism $\rho_z$.
By the description of
the isomorphism $\rho_5$ in Remark \ref{rem:rho and MPL}, 
$j_3(l^{(i)}_{1\to w})$ corresponds to a linear combination of
$\Li_{\bf k}(wz^{-1},1,\dots,1)$ with multi-indices $\bf k$
($\bf k$ can be empty).
Hence it lies on  $j_2(\A_z^0)$.

While, by Remark \ref{rem:MPL presentation of B and LST},
$i_3(m^{(i)})$ corresponds to a linear combination of
$\Li_{\bf h}(w,1,\dots,1)$ with multi-indices $\bf h$
($\bf h$ can be empty).
By $l\in\A^0_z$,  ${\bf k}_1$ is admissible 
(that means ${\bf e}_{{\bf k}_1}\in\A^0$) or empty.
\begin{itemize}
\item When ${\bf k}_1$ is admissible,  recursive differentiation of \eqref{eqn:rho=Li} shows that all $\bf h$ appearing are admissible.
Whence $i_3(m^{(i)})$ is always in $j_2(\A_z^0)$.
\item When ${\bf k}_1$ is empty, a non-admissible index $\bf h$ might occur
to contribute in $i_3(m^{(i)})$ for some $i$.
However such contribution is cancelled out
when we consider the summation $\sum_i j_3(l^{(i)}_{1\to w})\sha i_3(m^{(i)})$.
It is because
recursive differentiation of \eqref{eqn:rho=Li} implies
$j_3(l^{(i)})\in\A_z\cdot (e_1-e_w)$ for such $i$.
So a nontrivial contribution in the summation 
occurs  only in the case of 
$i_3(m^{(i)})\in j_2(\A_z^0)$.
\end{itemize}
Therefore we have $\sum_i j_3(l^{(i)})\sha i_3(m^{(i)})\in j_2(\A_z^0)$.
\end{proof}

Then we put
$$
\tilde l:=j_2^{-1}\left(\sum\nolimits_i j_3(l^{(i)}_{1\to w})\sha i_3(m^{(i)})\right)\in\A_z^0.
$$

\begin{lem}\label{lem:construction}
For $l\in\A^0_z$, we  have $l-\tilde l\in\I_\ST$.
\end{lem}

\begin{proof}
By \eqref{ISTA0} and the previous lemma,
it is enough to prove $j_2(l-\tilde l)\in \tilde\I_\ST$.
Since
$$
j_2(l-\tilde l)=\sum\nolimits_i j_3(l^{(i)}-l^{(i)}_{1\to w})\sha i_3(m^{(i)}),
$$
we have $l-\tilde l\in\I_\ST$ by Proposition \ref{basis presentation}.
\end{proof}

%
%

\subsection{Computations of the pairing}
The following lemma is auxiliary to prove Theorem \ref{thm:main 2}.
\begin{lem}\label{lem:2cycle}
Let 
$\varphi\in \exp [\hat{\mathfrak f}_2,\hat{\mathfrak f}_2]$
satisfies  the confluence relation.
Then it also satisfies the 2-cycle relation $\varphi(f_0,f_1)\varphi(f_1,f_0)=1$.
\end{lem}

\begin{proof}
Let $\varphi$ be as above.
Since  it is shown in
\cite{HS} Theorem 28 that the confluence relations imply the duality relation
(cf. Theorem \ref{thm:HS confluence-->double shuffle}),
we have $S_U(\varphi(f_1,f_0))=\varphi(f_0,f_1)$
with the antipode $S_U$ of $U{\mathfrak f}_2$.
While since $\varphi$ is group-like, we have $S_U(\varphi(f_1,f_0))=\varphi(f_1,f_0)^{-1}$.
Thus we get the claim.
\end{proof}

The proof of Theorem \ref{thm:main 2} goes as follows:
Assume that $\varphi\in \exp [\hat{\mathfrak f}_2,\hat{\mathfrak f}_2]$ satisfies the confluence relation, which is also equivalent to saying that 
$<\lambda(l) \mid\varphi>=0$ for any $l\in{\mathcal I}_{\ST}$.
Then for any $l\in{\mathcal I}_{\ST}$, we have
\begin{align}\label{eq: j2(l)P}
0&=<j_2(l)|\varphi_{243}^{-1}\varphi_{215}\varphi_{534}>
=<j_2(l)|\varphi_{342}\varphi_{215}\varphi_{534}> \\
\notag
&
=<j_2(l)|\varphi_{153}\varphi_{342}\varphi_{215}\varphi_{534}>.
\end{align}
Here
\begin{itemize}
\item
The first equality follows from the key formula in Proposition \ref{prop: key formula}.
\item
The second one is due to Lemma \ref{lem:2cycle}.
\item
The last one follows from the second equality of \eqref{eqn:on i2}, 
\eqref{frequent equation1}, \eqref{frequent equation2},
and the same arguments to the proof of 
\eqref{eqn:j3 315}
because
$j_2(l)\in \tilde \I_\ST$ is of the form with 
$\sum_i j_3(l^{(i)}-l^{(i)}_{1\to w})\sha i_3(m^{(i)})$
by Proposition \ref{basis presentation}.
\end{itemize}

Set $P'_{15342}=\varphi_{153}\varphi_{342}\varphi_{215}\varphi_{534}
\in \exp \widehat{\mathfrak P}_5$.
Let 
$\l=\rho_z^{-1}(\ell_{{\bf k}_1,{\bf k}_2,\dots,{\bf k}_N})\in\A^0_z$
with $N\neq 1$.
Then we have
\begin{align}\label{eq:j2lP13542}
<j_2(l)|&P'_{15342}>
=<j_2(\tilde l)|P'_{15342}> \\ \notag
&=\sum\nolimits_i<j_3(l^{(i)}|_{1\to w})\ |P'_{15342}>\cdot
<i_3(m^{(i)})|P'_{15342}> \\
&=\sum\nolimits_i<i_4(n^{(i)})\ |P'_{15342}>\cdot
<i_3(m^{(i)})|P'_{15342}>  
=0 \notag
\end{align}
when the decomposition of $j_2(l)$ is given by 
$j_2(\l)=\sum_i j_3(l^{(i)})\sha i_3(m^{(i)})$.
Here
\begin{itemize}
\item The first equality is by \eqref{eq: j2(l)P} as we have $l-\tilde l\in \I_\ST$
by Lemma \ref{lem:construction}.
\item The second one is by \eqref{frequent equation1}.
\item In the third one, $n^{(i)}\in\A^1$ is chosen to be
$i_4(n^{(i)})=j_3(l^{(i)}|_{1\to w})$:
Such an element always exists because
$j_3(l^{(i)}|_{1\to w})$
corresponds to a linear combination of
$\Li_{\bf k}(wz^{-1},1,\dots,1)$'s,
which span $i_4(\A^1)$, under $\rho_5$.
\item The last one follows from \eqref{eqn:on i4}
because we have 
$\pr_4(P'_{15342})=1$
by Lemma \ref{lem:2cycle}
and $\deg n^{(i)} > 0$ by $N>1$.
\end{itemize}


While by Lemma \ref{lem:2cycle}, we have $\pr_2(P'_{15342})=1$.
Since the kernel of the restriction 
$\pr_2|_{\widehat{{\mathfrak P}_5}}:\widehat{{\mathfrak P}_5}\to \widehat{{\mathfrak f}_2}$
is $\widehat{\frak f}_3(t_{21},t_{23},t_{24})$
and $P'_{15342}$ is in $\exp \widehat{\mathfrak P}_5$, we see that  
 $P'_{15342}\in \exp \hat{\frak f}_3(t_{21},t_{23},t_{24})$.
Here $\hat{\frak f}_3(t_{21},t_{23},t_{24})$ is the Lie subalgebra of $\hat{\frak P}_5$
freely generated by the three elements.
Therefore $P'_{15342}$ is described as 
$$P'_{15342}=r_2(P')$$
for a $P'\in\exp\widehat{\mathfrak f_z}$.

By \eqref{eq:j2lP13542},
we  have $<j_2(l)|P'_{15342}>=0$ for 
$\l=\rho_z^{-1}(\ell_{{\bf k}_1,{\bf k}_2,\dots,{\bf k}_N})\in\A^0_z$
with $N\neq 1$.
By $r_2\circ j_2=\id$ and \eqref{eqn:on r2},
we have
$$<l|P'>=<j_2(l)|P'_{15342}>=0$$
for such $l\in\A_z$. 
By $\varphi\in \exp [\hat{\mathfrak f}_2,\hat{\mathfrak f}_2]$,
there is no linear terms in $P'_{15342}\in\widehat{U\frak P_5}$.
Thus we have 
$P'\in\exp \hat{\frak f}_2\subset \exp \hat{\frak f}_z$.
Considering the case of $N=1$, that is,
$\l=\rho_z^{-1}(\ell_{{\bf k}_1})$,
we have
$$
<l|P'>=<j_2(l)|P'_{15342}>=<i_3(l)|P'_{15342}>=<l|\pr_3(P'_{15342})>
=<l|\varphi>
$$
by $r_2\circ j_2=\id$, \eqref{eqn:on r2},
$j_2(l)=i_3(l)$ for such $l$
and the second equality of \eqref{eqn:on i2}.
Therefore we have $P'=\varphi$, which says
$P_{15342}'=\varphi_{124}$.
So we have 
$\varphi_{153}\varphi_{342}\varphi_{215}\varphi_{534}\varphi_{421}=1$
by Lemma \ref{lem:2cycle}.
By considering the action of $(25)\in\frak S_5$ on $\widehat{U\frak P_5}$, we get
$\varphi_{123}\varphi_{345}\varphi_{512}\varphi_{234}\varphi_{451}=1$,
whence we obtain the pentagon equation for $\varphi$.
\qed
%

\bigskip
By combining  Theorem \ref{thm:main 1} and Theorem \ref{thm:main 2},
we settle the proof of Theorem \ref{thm:main}.

\bigskip
{\it Acknowledgments.} The author has been supported by JSPS KAKENHI JP18H01110. 
He thanks 
B. Enriquez for giving him comments on the earlier version of this paper
and S. Yasuda for informing him of Suzuki's master thesis (\cite{S})
partially related to this paper.
He is also grateful to the referee who gave many valuable comments to
improve the paper.

\end{document}